\documentclass[a4paper,12pt,reqno]{amsart}
\usepackage{amssymb}
\usepackage{mathtools}
\usepackage[english]{babel}
\usepackage{epsfig}
\usepackage{xcolor}

\usepackage{mathptmx}
\usepackage{amsmath,amsfonts,amssymb}
\usepackage{mathtools}
\usepackage{mathrsfs}
\usepackage[all]{xy}
\usepackage{graphicx}
\usepackage{latexsym}
\usepackage{verbatim}
\usepackage[pagewise]{lineno}

\numberwithin{equation}{section}

\newcommand{\C}{\mathbb C}

\newcommand{\Aut}{{\sf Aut}(\mathbb D)}

\newcommand{\bl}{\bold}

\newcommand{\overbar}[1]{\mkern 1.5mu\overline{\mkern-1.5mu#1\mkern-1.5mu}\mkern 1.5mu}

\def\id{{\sf id}}

\def\Aut{{\sf Aut}}

\def\1#1{\overline{#1}}
\def\2#1{\widetilde{#1}}
\def\3#1{\widehat{#1}}
\def\4#1{\mathbb{#1}}
\def\5#1{\frak{#1}}
\def\6#1{{\mathcal{#1}}}
\usepackage{enumitem}

\newcommand{\mcite}[1]{\csname b@#1\endcsname}

\theoremstyle{plain}

\def\id{{\sf id}}

\def\Aut{{\sf Aut}}

\def\Fix{\operatorname{Fix}}
\def\Deck{\operatorname{Deck}}

\def\mb{\mathbb}

\newtheorem{theorem}{Theorem}[section]
\newtheorem{lemma}[theorem]{Lemma}
\newtheorem{proposition}[theorem]{Proposition}
\newtheorem{corollary}[theorem]{Corollary}

\theoremstyle{definition}
\newtheorem{definition}[theorem]{Definition}

\theoremstyle{remark}
\newtheorem{remark}[theorem]{Remark}

\numberwithin{equation}{section}

\title[$2$-proper holomorphic images and Cartan Domains]{$2$-proper holomorphic images of classical Cartan domains}

\author[G. Ghosh]{Gargi Ghosh$^\dag$}
\address{G. Ghosh: Silesian University in Opava, 746 01, Czech Republic} 
\curraddr{Institute of Mathematics, Faculty of Mathematics and Computer Science, Jagiellonian University, \L ojasiewicza 6, 30-348 Krak\'ow, Poland}\email{gargi.ghosh@uj.edu.pl}

\author[W. Zwonek]{W\l odzimierz Zwonek}
\address{W. Zwonek: Institute of Mathematics, Faculty of Mathematics and Computer Science, Jagiellonian University, \L ojasiewicza 6, 30-348 Krak\'ow, Poland} \email{wlodzimierz.zwonek@uj.edu.pl}


\thanks{\textcolor{black}{This research is part of the project No. 2022/45/P/ST1/01028 co-funded by the
National Science Centre and the European Union Framework Programme for Research
and Innovation Horizon 2020 under the Marie Skłodowska-Curie grant agreement No.
945339.} $^\dag$Partially supported by postdocotoral fellowship from Silesian University in Opava under GA CR grant no. 21-27941S. }

\long\def\REM#1{\relax}

\begin{document}
\maketitle

\selectlanguage{english}
\begin{abstract} 

Motivated by the way two special domains, namely the symmetrized bidisc and the tetrablock, could be defined as the images of $2$-proper holomorphic images of classical Cartan domains, we present a general approach to study $2$-proper holomorphic images of bounded symmetric domains. We show some special properties of $2$-proper holomorphic maps (such as the construction of some \textcolor{black}{involutive} automorphisms etc.) and enlist  possible domains (up to biholomorphisms) which arise as $2$-proper holomorphic images of bounded symmetric domains. This leads us to a consideration of a new family of domains $\mathbb L_n$ for $n\geq 2.$ Let $L_n$ be an irreducible classical Cartan domain of type $IV$ (Lie ball) of dimension $n$ and $\Lambda_n : L_n \to \Lambda_n (L_n):=\mathbb L_n$ be the natural  proper holomorphic mapping of multiplicity $2.$ It turns out that $\mathbb L_2$  and $\mathbb L_3$ are biholomorphic to the symmetrized bidisc and the tetrablock, respectively. In this article, we study function geometric properties of the family $\{\mathbb L_n : n \geq 2\}$ in a unified manner and thus extend results of many earlier papers on analogous properties of the symmetrized bidisc and the tetrablock. We show that $\mathbb L_n$ cannot be exhausted by domains biholomorhic to some convex domains. Any proper holomorphic self-mapping of $\mathbb L_n$ is an automorphism for $n \geq 3.$ Moreover, the automorphism group $\Aut(\mathbb L_n)$ is isomorphic to $\Aut( L_{n-1})$ and $\mathbb L_n$ is inhomogeneous for \textcolor{black}{$n\geq2$}. Additionally, we prove that $\mathbb L_n$ is not a Lu Qi-Keng domain for $n \geq 3.$

\end{abstract}

\section{Introduction} 
Let $D_i, \, i=1,2,$ be two domains in $\mb C^n$ and $\pi : D_1 \to D_2$ be a proper holomorphic mapping of multiplicity $m.$ Our quest of understanding the geometric structure of $D_2$ for `nice' convex $D_1$ begins for $m=2$ case. A simple yet significant observation is that a proper holomorphic mapping with $m=2$ is always a Galois covering (in appropriate sense). However, the situation changes drastically even for $m=3,$ for example, there exists a Blaschke product of multiplicity $3$ from unit disc to itself which is not Galois. 

We start with a bounded symmetric domain $D$ in $\mb C^n.$ Suppose that $\pi:D\to \pi(D)\subset\mathbb C^n$ with $\pi(D)$ being a domain, is a proper holomorphic mapping with multiplicity $2.$ Then as a consequence of \cite[Theorem 3]{Got 1969}  and \cite{Mes 1972} we provide a list of possible $\pi(D)$ (up to biholomorphisms). 
Apart from the symmetrized bidisc, the tetrablock and the complex ellipsoids, that list contains a not so well-studied class of domains, say $\mathbb L_n.$ In this article, we focus on function geometric properties of $\mathbb L_n$ which extend our understanding of proper holomorphic mappings of multiplicity $2$. 

Let $L_n$
 be a classical Cartan domain of type IV in $\mathbb C^n$ and $\Lambda_n : L_n \to \Lambda_n (L_n):= \mathbb L_n$ be a proper holomorphic mapping of multiplicity $2$ defined by
 \begin{eqnarray*}
     \Lambda_n(z_1,z_2,\ldots,z_n)=(z_1^2,z_2,\ldots,z_n).
\end{eqnarray*} 
We note that $\mathbb L_2$ and $\mathbb L_3$ are biholomorphic to the symmetrized bidisc $\mathbb G_2$ and the tetrablock $\mathbb E,$ respectively (compare \cite{Sie 1950}, \cite[Remark, p. 139]{Coh-Col 1994}, \cite[Lemma 3]{Coh-Col 1994} and see Corollary \ref{bimap2} for explicit formulae of such biholomorphisms). Originating from $\mu$-synthesis, the domains $\mb G_2$ and $\mb E$ turn out to be topics of intense research in geometric function theory in their own right, for instance, these hyperconvex domains are first two examples such that none of those is biholomorphically equivalent to any convex domain but Lempert theorem holds on both domains \cite{Agl-You 2004, Edi 2004, Cos 2004, Edi-Kos-Zwo 2013}. 
Recently, an extensive study on both domains has been carried out in operator theory as well, to mention a few, see \cite{Bha-Pal-Roy 2012, Agl-You 2000, Agl-You 2001, Bis-Shy 2013, Agl-You 2003}  and the references therein. 
Therefore, in order to put our main results in the right perspective we describe the basic function geometric properties of $\mathbb L_n$ in a unified manner for all $n\geq 2$.
\begin{enumerate}[leftmargin=*]
    \item[1.] It is proved that $\mathbb L_n$ cannot be exhausted by domains biholomorhic to some convex domains in \cite{Edi 2004} for $n=2$ and in \cite{Edi-Kos-Zwo 2013} for $n=3.$ We show in Theorem \ref{convex} that it is true for all $\mathbb L_n$ for $n\geq 2.$
    \item[2.] In \cite{Ale 1977}, Alexander showed that a proper holomorphic self-mapping of the unit ball in $\mb C^n$ is an automorphism for $n>1.$ We say that a domain has {\it Alexander type property } if it has no proper holomorphic self-mapping except automorphisms. Unlike $\mathbb L_2,$ the domain $\mathbb L_3$ has Alexander type property \cite{Edi-Zwo 2005, Kos 2011}. In Proposition \ref{alexander}, we prove that $\mathbb L_n$ has Alexander type property for $n\geq 3.$
    \item[3.] Any biholomorphic self-mapping of $\mb G_2$ (equivalently, of $\mathbb L_2$) can be determined by an automorphism of the locus set of the symmetrization map (which is biholomorphic to the unit disc) \cite{Edi-Zwo 2005, Edi 2004b}. A similar observation is made for the tetrablock (equivalently, for $\mathbb L_3$) in \cite{You 2008}. We show that an analogous phenomenon holds for $\mathbb L_n$ for all $n\geq 2$. To determine the group of automorphisms of $\mathbb L_n,$ denoted by $\Aut(\mathbb L_n)$, we first construct automorphisms of $\mathbb L_n$ extending the ones of the locus set $\{0\}\times L_{n-1}$ of $\Lambda_n$ and then we show that there are no other automorphisms. Thus Theorem \ref{theorem:automorphisms} states the following: 
    $$\Aut(L_{n-1}) \cong \Aut(\mathbb L_n),\,\,n\geq 2.$$
    Therefore, the explicit description of $\Aut(L_{n-1})$ that can be found in \cite{Hua 1963} gives a complete form of $\Aut(\mathbb L_n)$.
    \item[4.] If the Bergman kernel $K_D$ of the domain $D\subset\mathbb C^n$ has no zero in $D\times D$, we say the domain $D$ is a {\it Lu Qi-Keng domain} \cite{Boa 2000}. While $\mathbb L_2$ is a Lu Qi-Keng domain \cite{Edi-Zwo 2005}, Trybula showed that the result is otherwise for $n= 3$ \cite{Try 2013}. We further show in Proposition \ref{Lu Qi-Keng} that $\mb L_n$ is not a Lu Qi-Keng domain for $n\geq 3.$
\end{enumerate}  

We emphasize that the approach adopted by us is much simpler and yet much more general. Moreover, it shows that the tedious work conducted in many papers for the cases $n=2,3$ can be obtained in a unified manner for arbitrary dimensions. Also, this approach leads us to numerous open questions for further exploration. 
\textcolor{black}{For instance, after we put this article on arXiv, A. Edigarian showed in \cite{Edi 2023} that the domain $\mathbb L_n$ is always a Lempert domain, which is earlier proven for the tetrablock and the symmetrized bidisc (that is the cases $n=2 $ and $3$). Another point of interest is to establish a relation between the equality of invariant functions on $\mathbb L_n$ and the problem of the complete understanding of the $3$-point Nevanlinna-Pick problem in classical Cartan domains (or at least in irreducible classical Cartan domains of type IV)}. The motivation for preceding problem comes from  \cite{Kos 2015} and \cite{Kos-Zwo 2018} where $3$-point Nevanlinna-Pick problem for the unit polydisc and the unit Euclidean ball are discussed, respectively.  Moreover, since the first attitude towards studying of $\mathbb G_2$ and $\mathbb E$ comes from $\mu$-synthesis, it is desirable to find the roots for studying the domains $\mathbb L_n$ for $n\geq 2$ in the very same theory. Lastly, the unified approach of this article can be reenacted to study operators on function spaces on $\mathbb G_2$ and $\mathbb E$ as well. Therefore, we expect that the operator theory on domains $\mathbb L_n$ play a pervasive role in that direction. 

\section{Proper holomorphic mappings with multiplicity two}
\subsection{Preliminaries}
Let $D$ and $G$ be two domains in $\mathbb C^n.$ A holomorphic map $\pi:D\to G$ is said to be {\it proper} if $\pi^{-1}(K)$ is a compact subset of $D$ for every compact $K\subset G.$ We start by recalling some basic properties of proper holomorphic mappings which are of our interest. 
\begin{theorem}\cite[Chapter 15]{Rud 1980} Let $\pi:D\to G$ be a proper holomorphic mapping. Then 
\begin{enumerate}
    \item[(i)] $\pi$ is onto and
    \item[(ii)] there exists a positive integer $m$ such that $\pi:D\setminus\pi^{-1}(\pi(\mathcal J_\pi))\to G\setminus\pi(\mathcal J_\pi)$  is a holomorphic covering with
\begin{align*}
\text{cardinality of }\pi^{-1}(w)&=m,\; \,\,w\in G\setminus\pi(\mathcal J_\pi) \text{ and, }\\
\text{cardinality of }\pi^{-1}(w)&<m,\; \,\,w\in\pi(\mathcal J_\pi),
\end{align*}
where $\mathcal J_\pi:=\{z\in D:\operatorname{det}\pi^{\prime}(z)=0\}.$
\end{enumerate} 
\end{theorem}
We refer to $m$ as {\it the multiplicity of $\pi$} and the mapping $\pi$ as {\it $m$-proper} holomorphic mapping. The set $\mathcal J_\pi$ is said to be {\it the locus set of $\pi.$} The sets $\mathcal J_\pi$ and $\pi(\mathcal J_\pi)$ are analytic sets of dimension smaller than $n$. The points of $\pi(\mathcal J_\pi)$ are called {\it critical values of $\pi$} and the points of $G\setminus\pi(\mathcal J_\pi)$ are called {\it regular values of $\pi$}.
 
\subsubsection{Deck automorphisms and factorization}

We first introduce a notion of isomorphism between two proper holomorphic mappings.
\begin{definition}\label{iso}
We say that two proper holomorphic mappings $\pi_j:D_j\to G_j$, $j=1,2$ are {\it isomorphic} if there exist two biholomorphisms $a:D_1\to D_2$ and $b:G_1\to G_2$ such that $$b\circ \pi_1=\pi_2\circ a.$$ 
\end{definition}
For instance, for every $a\in\Aut(D),$ the proper holomorphic mapping $\pi\circ a:D\to G$ is isomorphic to the proper holomorphic mapping $\pi:D\to G.$

Suppose that $a:D_1\to D_2$ is a fiber-preserving biholomorphism, that is, for every $z\in G_1$ there is a $w\in G_2$ such that $a(\pi_1^{-1}(z))\subset\pi_2^{-1}(w)$, then we get the existence of a biholomorphism $b:G_1\to G_2$ which satisfies the condition in Definition \ref{iso}. In fact, the mapping $b$ can be well-defined by $b(z):=w$. The properties of proper holomorphic mappings allow us to deduce that $b$ is continuous and holomorphic off the set of critical values of $\pi_1$. Riemann removability theorem assures that $b$ is holomorphic on $G_1$. Then the \textcolor{black}{biholomorphicity} of $b$ is evident.

Recall that a holomorphic automorphism $\rho$ of $D$ is called a {\it deck automorphism} of the proper holomorphic mapping $\pi:D\to G$ if $\pi\circ \rho=\pi$. If the group of deck automorphisms $\operatorname{Deck}(\pi)$ acts transitively on $\pi^{-1}(w)$ for every $w\in G$ then the proper holomorphic mapping $\pi$ is called {\it Galois} \cite{Sza 2009}. Suppose that $\rho$ is a deck automorphism of $\pi,$ then the conjugate $a^{-1}\circ\rho\circ a$ is a deck automorphism of $\pi\circ a,$ $a\in\Aut(D).$


Below we formulate a simple yet significant result for $2$-proper holomorphic mappings.

\begin{proposition}\label{proposition:idempotent-automorphism} Suppose that $\pi:D\to G$ is a $2$-proper holomorphic mapping.
Then there exists an element $g\in\Aut(D),$ $g \neq \id,$ such that $g\circ g = \id,$ where $\id$ denotes
the identity of $\Aut(D).$ Moreover, $g|_{\mathcal J_\pi}$ is the identity.
\end{proposition}
\begin{proof} 
First note that for a $2$-proper holomorphic mapping, $\pi^{-1}(\pi(\mathcal J_\pi)) = \mathcal J_\pi.$  For a regular value $w\in G,$ the set $\pi^{-1}(w)$ consists of two elements, say $\pi^{-1}(w)=\{z_1,z_2\}.$ We define $g : D\setminus \mathcal J_\pi \to D\setminus \mathcal J_\pi$ by $g(z_j):=z_{3-j}$, $j=1,2.$ This formula gives a well-defined holomorphic self-mapping on the complement of the analytic set $\mathcal J_\pi$ in $D.$ For $w\in\pi(\mathcal J_\pi),$ we have $\pi^{-1}(w)=\{z\}.$ We define $g$ on $\mathcal J_\pi$ as the identity. Then $g$ is holomorphic off $\mathcal J_\pi$ and continuous on $D$. Riemann removability theorem concludes that $g$ is holomorphic on $D$. From the definition, it follows that $g\circ g=\id$ which finishes the proof.
\end{proof}
Thus, the existence of $2$-proper holomorphic mapping on $D$ ensures the existence of a non-trivial automorphism group $\Aut(D)$ of the domain $D.$ Clearly, $g$ is a deck automorphism of $\pi$ and it transitively acts on each fibre. Hence, we have the following result. 
\begin{corollary}\label{galois} The $2$-proper holomorphic mapping $\pi : D \to G$ is Galois, $\Deck(\pi)=\{\id,g\}$ and $D/\operatorname{Deck}(\pi)$ is biholomorphic to the domain $G$.
\end{corollary}


\begin{remark} Recall that if the multiplicity $m>2,$ the proper holomorphic mapping is not necessarily Galois. For example, in \cite[Section 4.7, p. 711]{Rud 1982} a Blaschke product of degree $3$ is described which is not Galois. Also, in \cite[Example, p. 223]{Dini-Sel-Pri 1991} an example in the class of ellipsoids is provided for the same. On the other hand, under mild assumptions all the proper holomorphic mappings defined on the unit ball $\mathbb B_n$, $n>1$, are Galois \cite[Theorem 1.6]{Rud 1982}. Later, this result of Rudin is extended to the irreducible bounded symmetric domains in \cite{Mes 1988}.
\end{remark}





 \subsection{Sets of fixed points} 

We define the {\it set of fixed points} of a holomorphic self-mapping $g$ of $D$ by
\begin{equation*}
\operatorname{Fix} (g):=\{x\in D:g(x)=x\}.
\end{equation*}
Proposition~\ref{proposition:idempotent-automorphism} shows the existence of a holomorphic self-mapping $g$ on $D$ such that $\operatorname{Fix}(g)=\mathcal J_\pi$  which justifies our interest in the geometry of $\operatorname{Fix}(g)$.  

Before we proceed to further consideration, we recall some basic notions and results on the holomorphically invariant functions and Lempert domains. Lempert domains are of our interest since we can say more about the structure of the set of fixed points in this class of domains.
\subsubsection{(Holomorpically) invariant functions and Lempert domains}
For a domain $D\subseteq\mathbb C^n,$ the {\it Lempert function} $l_D$ is defined as following: for $w,z\in D,$ 
\begin{equation*}
    l_D(w,z):=\inf\{p(\lambda_1,\lambda_2):\exists f:\mathbb D\to D \text{ holomorphic and such that } f(\lambda_1)=w,\; f(\lambda_2)=z\},
\end{equation*} where $p$ denotes the {\it Poincar\'e distance} on the unit disc $\mathbb D\subset\mathbb C$.
{\it The Carath\'eodory pseudo-distance} is defined by
\begin{equation*}
    c_D(w,z):=\sup\{p(F(w),F(z)):F:D\mapsto\mathbb D \text{ is holomorphic}\}.
\end{equation*}
We say that a taut domain $D\subset\mathbb C^n$ is {\it a Lempert domain} if the identity $l_D\equiv c_D$ holds. In particular, any pair of distinct points are connected by a {\it complex geodesic} $f:\mathbb D\to D$ that has a {\it left inverse} $F:D\to\mathbb D$, that is, $F\circ f=\id_{\mathbb D}$. For example, bounded convex domains, strongly linearly convex domains, the symmetrized bidisc and the tetrablock are Lempert domains \cite{Lem 1981,Lem 1982,Lem 1984,Agl-You 2001,Cos 2004,Agl-You 2006,Abo-You-Whi 2007,Edi-Kos-Zwo 2013}. Additionally, a good reference monograph for basic properties of holomorphically invariant functions is \cite{Jar-Pfl 2013}.

Now we recall properties which illustrate the structure of $\operatorname{Fix}(g)$ more intrinsically for certain domains. 
\begin{remark}\label{remark:fix} There is a huge literature on the properties of sets of fixed points. We enlist below the results which are mostly due to J. P. Vigue and can be found, for instance in \cite{Vig 1985} (see also \cite{Aba-Vig 1991}) or are direct consequences of the Lempert theory.
\begin{enumerate}
\item[1.] If $D$ is strongly linearly convex or the Lempert domain with the uniqueness of complex geodesics, for example, strictly convex domains, the symmetrized bidisc then $\operatorname{Fix}(g)$ is {\it totally geodesic}, that is, any complex geodesic joining two points in the set must lie in the set completely.
\item[2.] Consequently, if $D$ is strongly linearly convex in $\mathbb C^2,$ then $\operatorname{Fix}(g)$ is either an empty set, a point, a complex geodesic or the whole $D.$
\item[3.] If $D$ is a bounded convex domain then $\operatorname{Fix}(g)$ is {\it weakly totally geodesic}, that is, for any $w,z\in \operatorname{Fix}(g)$ we can find a complex geodesic joining $w$ and $z$ such that its graph is contained in $\operatorname{Fix}(g).$ Moreover, $\operatorname{Fix}(g)$ is a submanifold and retract.
\item[4.] If $g:\mathbb B_n\to\mathbb B_n$ then $\operatorname{Fix}(g)$ is the intersection of the affine subspace with $\mathbb B_n$.
\end{enumerate}
\end{remark}

Now we go back to our previous setting to use the results mentioned in Remark~\ref{remark:fix}. Using Proposition~\ref{proposition:idempotent-automorphism}, we conclude that for every $2$-proper holomorphic mapping $\pi : D \to G$ there exists an \textcolor{black}{involution} $g \in \Aut(D)$ such that $\operatorname{Fix}(g) = \mathcal J_\pi=\{z\in D: \det \pi^{\prime}(z)=0\},$ the locus set of $\pi.$
Suppose that $D$ is a bounded convex domain. Then a direct consequence of Remark~\ref{remark:fix} is that the locus set $\mathcal J_\pi$ of a $2$-proper holomorphic mapping $\pi:D\to G$ is weakly totally geodesic, a retract and a submanifold. However, a more can be stated:

\begin{proposition} Let $\pi:D\to G$ be a $2$-proper holomorphic mapping, where $D$ is a bounded convex domain. Then $\pi(\mathcal J_\pi)$ is a retract in $G$. Consequently, it is a submanifold and weakly totally geodesic. Moreover, 
\begin{equation*}
c_D(z_1,z_2)=l_D(z_1,z_2)=c_G(\pi(z_1),\pi(z_2))=l_G(\pi(z_1),\pi(z_2)), \;\, z_1,z_2\in \mathcal J_\pi.
\end{equation*}
 \end{proposition}
\begin{proof} It is known that there exists a holomorphic retract $r:D\to \mathcal J_\pi$. Then $s:G\to\pi(\mathcal J_\pi)$ defined by $$s(w):=\pi\left(r\left(\frac{z+a(z)}{2}\right)\right),$$ is a retract in $G,$ where $a$ is the deck automorphism and $\pi^{-1}(w)=\{z,a(z)\}$. Thus, $\mathcal J_\pi$ is a submanifold. Moreover, $\pi|_{\mathcal J_\pi}$ preserves complex geodesics. More precisely, if $f:\mathbb D\to D$, where $f(\lambda_j)\in\mathcal J_\pi$, $j=1,2$, $\lambda_1\neq\lambda_2$, is a complex geodesic with the left inverse $F$ then the map $\Phi:G\owns w\to F\left(r\left(\frac{z+a(z)}{2}\right)\right)$ is a left inverse to $\pi\circ f$. In particular, $\pi(\mathcal J_\pi)$ is weakly totally geodesic and for $z_1,z_2\in\mathcal J_\pi,$
\begin{equation*}
c_D(z_1,z_2)=l_D(z_1,z_2)=c_G(\pi(z_1),\pi(z_2))=l_G(\pi(z_1),\pi(z_2)).
\end{equation*}
\end{proof}

In the result below we give a class of domains (including convex ones) ensuring the non-triviality of $\Fix(g)$ for \textcolor{black}{involution}s $g \in \Aut(D).$
 
\begin{proposition}\label{proposition:fix-non-empty} Let $g:D\to D$ be such that $g\circ g=\id$, where $D$ is a Lempert domain with uniquely determined geodesics or a bounded convex domain. Then $\Fix(g)$ is not empty.
\end{proposition}
\begin{proof} 




Choose $a\in D$ and put $b:=g(a)$. For our proof assume that $a\neq b$. Define by $K$ the set of all mid-points of complex geodesics joining $a$ and $b$ (in the image of any complex geodesic joining $a$ and $b$ there is exactly one point $z$ lying in the complex geodesic such that $2l_D(a,z)=2l_D(z,b)=l_D(a,b)$). Evidently, $g(K)=K$.

If $D$ is a Lempert domain with uniquely determined complex geodesics then $K$ consists of only one point and $K\subset\Fix(g)$.

If $D$ is a bounded convex domain then $K$ is convex and then Brouwer theorem provides the existence of a point $z\in K$ such that $g(z)=z$.
\end{proof}

 \begin{remark} On the other hand, the non-triviality of $\Fix(g)$ for any \textcolor{black}{involution} $g \in \Aut(D)$ does not always hold for non-convex domains. For example, consider $g:D \to D$ defined by $\lambda\mapsto-\lambda$, where $D\subset\mathbb C$ is an annulus centered at the origin. Note that $\operatorname{Fix}(g)$ is empty.
\end{remark}
 

We end this subsection with a list of examples in order to illustrate the structure of fixed points of the \textcolor{black}{involutive} automorphisms of some well-known domains. This allows us to characterize all possible $2$-proper holomorphic mappings defined on such domains up to isomorphisms.




\subsubsection{Dimension one - examples}
In dimension $1,$ we have the following equivalence classes for $2$-proper holomorphic mappings.
 
We claim that all $2$-proper holomorphic mappings on the unit disc $\mb D$ are isomorphic to $F : \mb D \to \mb D$ such that $ F(z)= z^2.$  Suppose that $\pi : \mb D \to G$ is a $2$-proper holomorphic mapping with ${\rm Deck}(\pi) = \{\id,a\}.$ To prove that $\pi$ and $F$ are isomorphic, we first note that $\operatorname{Fix}(a)$ is non-empty which essentially means here that it consists only of one point, say $\lambda_0.$ Since $a\circ a$ is the identity, for the automorphism $m$ of $\mathbb D$ with $m(0)=\lambda_0$ we have $m^{-1}\circ a \circ m(z)=-z=-\id(z)$ for $z\in \mb D.$ Therefore, without loss of generality we assume $\operatorname{Deck}(\pi)=\{\id,-\id\}$ and define the desired biholomorphism by the formula $G\owns w\to z^2\in\mathbb D$ where $\pi^{-1}(w)=\{z,-z\}$.


Suppose that $D$ is the annulus with radii $r<1<1/r,$ denoted by $A(r,1/r)$ and $\pi : A(r,1/r) \to G$ is a $2$-proper holomorphic mapping. Then $\pi$ has to be isomorphic to any of the following: 
\begin{enumerate}
    \item[1.] The mapping $\pi_1 : A(r,1/r) \to A(r^2,1/r^2)$ defined by $\pi_1( z)= z^2$ is isomorphic to $\pi$ if $\operatorname{Deck}(\pi)=\{\id,-\id\}.$ Here, $\Fix(-\id)$ is an empty set.
    \item[2.] the proper holomorphic mapping $\pi_2$ between $A(r,1/r)$ and the ellipse (and consequently by Riemann mapping theorem to the disc) defined by $\pi_2(z) = z+\frac{\omega}{z}$ is isomorphic to $\pi$ if $\operatorname{Deck}(\pi)=\{\id,g_\omega\}$, where $g_\omega(z)=\omega/z$ for some $|\omega|=1.$ In this case, $\Fix(g_\omega) = \{z : z^2 = \omega\}.$ Here, $\Fix (g_{\omega})$ consists of two points.
\end{enumerate}

\subsubsection{Dimension two - examples}\label{couext}
Here, we consider $2$-proper holomorphic mappings $\pi : D \to G$ where $D$ is either the bidisc $\mb D^2$ or the Euclidean ball $\mb B_2$ and $G$ is a domain in $\mathbb C^2$. We denote the deck automorphism group by ${\rm Deck}(\pi) = \{\id, g\}$ and provide a characterization of $\pi$ on the basis of $g.$

 First we claim that 
 $g$ can never be $-\id.$ Let us consider the {\it generalized Neil parabola} defined by $\{w\in\mathbb C^3:w_1w_2=w_3^2\}$ and a mapping $\pi_0 : D\to\{w\in\mathbb C^3:w_1w_2=w_3^2\}$ defined by
\begin{equation*}
 \pi_0 : z\to (z_1^2,z_2^2,z_1z_2),
\end{equation*} for $D=\mathbb B_2,\mathbb D^2.$ Then $\pi$ is isomorphic to $\pi_0$ and the associated $G$ is biholomorphic to the complex space $D/\{\id,-\id\}$ which can be realized as the generalized Neil parabola \textcolor{black}{and the set of fixed points of $-\id$ contains only the origin. Suppose that the generalized Neil parabola is biholomorphic to a domain in $\mathbb C^2$ then the dimension of the set of fixed points of $-\id$ has to be $2-1=1,$ (cf. Remark~\ref{remark:only-reflections}) which is not the case here. So the generalized Neil parabola is not biholomorphic to any domain in $\mathbb C^2.$ This is a contradiction since we assume $G$ to be a domain in $\mathbb C^2$.}  
We study the rest of $2$-proper holomorphic mappings separately on $\mathbb B_2$ and $\mathbb D^2.$ For $D=\mathbb D^2,$
\begin{enumerate}
\item[1.] if $g(z_1,z_2)=(-z_1,z_2)$ then $G$ is bihlomorphic to the bidisc and $\pi$ is isomorphic to $\pi_1 : \mb D^2 \to \mb D^2$ which is defined by $\pi_1(z_1,z_2)= (z_1^2,z_2).$ Here, $\Fix(g)=\{0\}\times\mathbb D$.
\item[2.] If $g_\omega(z_1,z_2)=(\overbar{\omega}^2 z_2,\omega^2 z_1)$ for some $|\omega|=1$ then $G$ is bihlomorphic to the symmetrized bidisc $\mathbb G_{2}$ and $\pi$ is isomorphic to $\pi_{2,\omega} : \mb D^2 \to \mathbb G_{2}$ which is defined by  $$ (z_1,z_2)\to (\overbar{\omega} z_1+\omega z_2,z_1z_2).$$ Here,
$\Fix(g_\omega)=\{(\overbar{\omega} \lambda,\omega\lambda):\lambda\in\mathbb D\}.$ 
\end{enumerate}
On the other hand, for the Euclidean ball we only have the case where $g(z_1,z_2)=(z_1,-z_2).$ Then the domain $G$ is represented by the complex ellipsoid $\{(z_1,z_2) : |z_1|^2+|z_2|<1\}$ and $\pi$ is isomorphic to $\pi_3 : (z_1,z_2)\mapsto(z_1,z_2^2).$ Also, $\Fix(g)=\mathbb D\times\{0\}.$ 

\subsection{Reflections} 
The above examples show that $D/{\rm Deck}(\pi)$ can be a complex space which is not necessarily a domain. However, if ${\rm Deck}(\pi)$ is a pseudoreflection group then $D/{\rm Deck}(\pi)$ is ensured to be a domain. In the characterization of $2$-proper holomorphic mappings from the classical Cartan domains, pseudoreflections of order $2$ (reflections) play the most important role which we describe in the next section.
\begin{definition}
{\it A pseudoreflection} on $\C^n$ is a linear homomorphism $\rho: \C^n \rightarrow \C^n$ such that $\rho$ is of finite order and the rank of $\id_n - \rho$ is 1. In particular, if $\rho$ is \textcolor{black}{of order $2$,} we refer to it by {\it reflection}.
\end{definition}
 Here, $\id_n$ denotes the identity operator from $\mb C^n$ to itself. Clearly, ${\rm ker}(\id_n-  \rho)$ is a hyperplane in $\mb C^n.$  In other words, the reflection $\rho$ is the identity on the hyperplane ${\rm ker}(\id_n-  \rho)$ and $\rho(v)=-v$ for some $v\neq 0.$
A group generated by pseudoreflections is called a pseudoreflection group. A pseudoreflection group $\Gamma$ acts on $\mb C^n$ by $\rho \cdot z = \rho^{-1}z$ for $\rho \in \Gamma$ and $z \in \mb C^n.$ Moreover, $\Gamma$ acts on the set of all complex-valued function on $\mb C^n$ by $\rho(f)(z) = f(\rho^{-1}\cdot z).$ We say $f$ is $\Gamma$-invariant if $\rho(f)=f$ for every $\rho \in \Gamma.$ 

Chevalley-Shephard-Todd theorem says that the ring of $\Gamma$-invariant polynomials in $n$ variables is equal to $\C[\theta_1,\ldots,\theta_n]$, where $\theta_i$ are algebraically independent
homogeneous polynomials. We refer to the mapping $\theta:=(\theta_1,\ldots,\theta_n) : \mb C^n \to \mb C^n$ by a \emph{basic polynomial} map associated to the group $\Gamma.$ Suppose that a domain $D \subseteq \mb C^n$ remains invariant under the action of $\Gamma.$ It is known that $\theta : D \to \theta(D)$ is a proper holomorphic mapping with the deck automorphism group $\Gamma$ \cite{Try 2013, Rud 1982}. Moreover, any proper holomorphic map $f : D \to D^\prime$ with the deck automorphism group $\Gamma$ is isomorphic to $\theta$ and $D^\prime$ is biholomorphic to $\theta(D)$ \cite{Gho 2021}. 

We observe in Example \ref{couext} that the deck automorphisms for $\pi_1$ and $\pi_{2,1}$ on $\mb D^2$ are conjugate to each other in the set of all $2\times 2$ complex matrices. However, they fail to be a conjugate to each other in $\Aut(\mb D^2)$ which motivates us to the following result. We describe this example in more detail in Remark \ref{nonex}. 

\begin{proposition}\label{proposition:conjugate}
    Suppose that $\Gamma_1$ and $\Gamma_2$ are two pseudoreflection groups and $\Gamma_2 = P\Gamma_1 P^{-1}.$ Also, suppose the domain $D$ is invariant under the action of $\Gamma_1$ and $P$ induces an automorphism of $D.$ Then the basic polynomial maps $\theta_1 : D \to \theta_1(D)$ and $\theta_2 : D \to \theta_2(D)$ are isomorphic. 
\end{proposition}

\begin{proof}
More precisely, we prove that
    $$\theta_2 \circ P = h \circ \theta_1$$ for some biholomorphism $h : \theta_1(D) \to \theta_2(D).$
    
    Note that $\theta_1 : D\to \theta_1(D)$ is a proper holomorphic mapping with the deck automorphism group $\Gamma_1,$ that is,
    $$\theta_1^{-1}\theta_1(z) = \{\sigma\cdot z : \sigma \in \Gamma_1\}.$$
    Then $\theta_1(D)$ is biholomorphic to $D/\Gamma_1$ \cite[Subsection 3.1.1]{Bis-Dat-Gho-Shy 2022}. Consider the map $F=\theta_1 \circ P^{-1}$ which has the following property :
    $$F^{-1}F(z) = \{P\sigma P^{-1}\cdot z : \sigma \in \Gamma_1\},$$ that is, the deck automorphism group of $F$ is $\Gamma_2.$ Also, we note that for $P$ being an automorphism of $D,$ $F$ is well-defined and $F(D)= \theta_1(D).$ Thus being factored by the same pseudoreflection group $\Gamma_2,$ both $F(D)$ and $\theta_2(D)$ are biholomorphic to $D/\Gamma_2$ by \cite[Subsection 3.1.1]{Bis-Dat-Gho-Shy 2022}. Therefore, there exists a biholomorphic mapping $h : \theta_1(D) \to \theta_2(D)$ such that $h \circ F = \theta_2$ that is $h \circ \theta_1 \circ P^{-1} = \theta_2$ which proves the claim.
\end{proof}

\begin{remark}\label{remark:conjugate} We provide another elementary way to see that the conjugate reflections give the same (up to biholomorpisms) images of proper holomorphic mappings without addressing \cite[Subsection 3.1.1]{Bis-Dat-Gho-Shy 2022}.  Let $\sigma_1:=\sigma$ and $\sigma_2:=P\sigma P^{-1}$ be two reflections such that the domain $D$ is invariant under $\sigma_j$, $j=1,2$ and $P$ be a linear isomorphism of $D.$ 
Let $R_j$ be the hyperplane such that $\sigma_j$ leaves it invariant. Then $P(R_1)=R_2$. Denote the non-trivial eigenvectors of $\sigma_j$ by $v_j$. Without loss of generality, we assume that $P(v_1)=v_2$. Then the basic polynomials $\theta_j$ can be defined as following. Let $A$ and $B$ be linear isomorphisms such that $A(R_1)=B(R_2)=\{0\}\times\mathbb C^{n-1}$, $A(R_1\cap D)=B(R_2\cap D)$ and $A(v_1)=B(v_2)=(1,0,\ldots,0) \in \mb C^n$. We put
\begin{equation*}
    \theta_1(z):=((A_1(z))^2,A_2(z),\ldots,A_n(z)),\; \theta_2(z):=((B_1(z))^2,B_2(z),\ldots,B_n(z)).
\end{equation*}
Then $\theta_1=\theta_2\circ P$ which gives that the representing domains of $D/\Gamma_1$ and $D/\Gamma_2$ are biholomorphic (identical for the choice we made here).
\end{remark}

\begin{remark}\label{nonex}
We explain the above result in details with examples for bidisc.
    \begin{enumerate}
        \item[1.] Recall that the proper holomorphic mapping $\pi_1 : \mathbb D^2\owns z\to (z_1,z_2^2)\in\mathbb D^2$ has the deck automorphism group $\{\id, diag(1,-1)\}$ 
        and  $\pi_{2,\omega} : \mathbb D^2\owns z\to (\overbar{\omega} z_1+\omega z_2,z_1z_2)\in\mathbb G_{2}$ has the deck automorphism group $\{\id,\begin{pmatrix}
            0& \omega^2\\
            \overbar{\omega}^2& 0
        \end{pmatrix}\}.$
        \item[2.]  Clearly, for $P_\omega = diag\{\omega,\overbar{\omega}\},$ $\begin{pmatrix}
            0& \omega^2\\
            \overbar{\omega}^2& 0
        \end{pmatrix} = P_\omega \begin{pmatrix}
            0& 1\\
            1& 0
        \end{pmatrix}P_\omega^{-1}$ and $P_\omega$ induces an automorphism of $\mb D^2.$ Then $\pi_{2,\omega} \circ P_\omega = \pi_{2,1}.$ Thus each element of $\{\pi_{2,\omega}: |\omega| = 1\}$ is isomorphic to $\pi_{2,1}.$
        On the other hand, $$\begin{pmatrix}
            0& 1\\
            1& 0
        \end{pmatrix}= P \begin{pmatrix}
            1& 0\\
            0& -1
        \end{pmatrix}P^{-1},$$ where $P = \begin{pmatrix}
            1& -1\\
            1& \phantom{-}1
        \end{pmatrix}$ can never induce an automorphism of $\mb D^2.$ So none of $\pi_{2,\omega}$ is isomorphic to $\pi_1.$ 
    \end{enumerate}
\end{remark}

\section{Classification of domains those are $2$-proper holomorphic images of the classical Cartan domains}
In this section, we generalize the method of producing two very special domains, namely the symmetrized bidisc and the tetrablock, that have been intensively studied in the last two decades and stimulated progress in different fields of mathematics, specially in the theory of several complex variables and operator theory, see \cite{Agl-You 2001,Agl-You 2006,Cos 2004,Abo-You-Whi 2007,Edi-Zwo 2005,Edi-Kos-Zwo 2013}, \cite{Agl-You 2000,Pal-Sha 2014,Bha 2014} and the references therein. There are many different ways of defining them. Though the original interest came from the $\mu$-synthesis, we recall the definition which suits the most with the theory developed in this article. Both of those can be defined as $2$-proper holomorphic images of some classical Cartan domains. First, we define {\it the symmetrized bidisc $\mathbb G_2$} as the $2$-proper holomorphic image of the bidisc $\mathbb D^2$ under the mapping $\pi$  (see \cite{Agl-You 2001}):
\begin{equation*}
\pi:(\lambda_1,\lambda_2)\to (\lambda_1+\lambda_2,\lambda_1\lambda_2),
\end{equation*}
whereas {\it the tetrablock } is defined as the image of the Cartan domain of type $III,$ $\mathcal R_{III}(2)$ (defined in Subsection \ref{car}) under the mapping (see \cite{Abo-You-Whi 2007})
\begin{equation*}
\Phi:A\to(a_{11},a_{22},\operatorname{det}A), \text{ where } A=\begin{pmatrix} a_{11} & a \\ a & a_{22}\end{pmatrix}.
\end{equation*}
Below we pursue this idea to produce many classes of domains which are the images of $2$-proper holomorphic images of Cartan domains and observe that in a quite natural sense, we could reduce those domains only to two new classes of domains. We show that those family of domains deliver similar geometric properties as the symmetrized bidisc and the tetrablock. Therefore, it is expected to attract attention from experts in at least two areas: several complex variables and operator theory. We begin with presenting some basic properties of classical Cartan domains.
   
\subsection{Classical Cartan domains.}\label{car} 
E. Cartan completely classified the irreducible bounded symmetric domains in \cite{Car 1935} up to a biholomorphic isomorphism. Those domains are referred as irreducible {\it classical Cartan domains} in this article. We reproduce that list  below. In the list, we omit the term `irreducible' since there is no ambiguity

 \begin{itemize}
\item $\mathcal R_I(m\times n) :$ The {\it classical Cartan domain of type $I$} consists of $m\times n$ complex matrices $A$ such that the matrix $\mathbb I_m-AA^*$ is positive definite, (that is, $\mathbb I_m-AA^*>0$ equivalently, $||A||<1$) where $A^*$ denotes the adjoint of $A$ and $\mb I_m$ denotes the identity matrix of order $m$.
\item $\mathcal R_{II}(n) : $  The {\it classical Cartan domain of type $II$} consists of the $n \times n$ skew-symmetric matrices such that $||A||<1.$
\item $\mathcal R_{III}(n) : $ The {\it classical Cartan domains of type $III$} is the set of $n \times n$ symmetric matrices such that $||A||<1.$
\item $\mathcal R_{IV}(n) : $ The {\it classical Cartan domains of type $IV$} (alternatively, {\it the Lie ball} $L_n$) is the following domain :
\begin{equation}
L_n:=\left\{z\in\mathbb B_n: \sqrt{\left(\sum_{j=1}^n|z_j|^2\right)^2-\left|\sum_{j=1}^nz_j^2\right|^2}<1-\sum_{j=1}^n|z_j|^2\right\}.
 \end{equation}
\end{itemize}
We refer to the Cartesian products of above mentioned domains as {\it classical Cartan domains}.
Additionally, there are two special irreducible symmetric bounded domains of dimensions $16$ and $27$ that we skip describing here in details. In the sequel, we collectively refer to these exceptional domains along with the irreducible classical Cartan domains by  {\it irreducible Cartan domains}.

Let denote $z\bullet z:=\sum_{j=1}^nz_j^2$ and $||z||:=\sqrt{\sum_{j=1}^n|z_j|^2}$ for  $z\in\mathbb C^n.$ A direct computation shows that $z\in L_n$ if and only if $||z||<1$ and $||z||^4-|z\bullet z|^2<(1-||z||^2)^2$ if and only if $||z||<1$ and $2||z||^2<1+|z\bullet z|^2.$

Recall that above domains are bounded, balanced and convex. There is some ambiguity in the numbering of types of classical domains in the literature. In the sequel, we follow the one presented above.

\subsection{Images of $2$-proper holomorphic mappings of classical domains} Having in mind the construction of the symmetrized bidisc and the tetrablock as the images of $2$-proper holomorphic mappings defined on classical Cartan domains $\mathbb D^2$ and $\mathcal R_{III}(2)$, respectively, we look at the possible images of other classical Cartan domains under $2$-proper holomorphic mappings. These mappings are (among others) determined by the set of fixed points of a deck automorphism (being \textcolor{black}{the involutive mapping}).
 
It is already known that any \textcolor{black}{involutive} automorphism of the classical Cartan domain has non-empty set of fixed points (cf. Proposition~\ref{proposition:fix-non-empty}). Additionally, since the group of automorphisms of classical domain is transitive, we assume that the deck automorphism $F$ of the given $2$-proper holomorphic mapping defined on the classical domain fixes the origin. Cartan theorem lets us conclude that $F$ is actually the linear mapping in this case. As we have already seen in nice domains (such as the bidisc and the Euclidean ball) if the set of fixed points of such $F$ is small (in the sense that it is of co-rank bigger than $1$) then the corresponding $2$-proper image of the domain is a complex space (and not a domain). Therefore, from this point to ensure that we are not leaving the category of domains of the Euclidean space, we assume that the set $\Fix(F)$ is a linear hyperspace. Then the group of automorphisms $\{\id,F\}$ factors the proper holomorphic mapping onto a domain as a consequence of the Chevalley-Todd-Shephard theorem \cite[Subsection 3.1.1]{Bis-Dat-Gho-Shy 2022}. In this situation, $F$ is {\it a reflection}, that is a linear \textcolor{black}{mapping of order $2$}, leaving a hyperplane invariant. Note that in the general situation the deck automorphism of a $2$-proper holomorphic mapping between domains in $\mathbb C^n$ can generate an \textcolor{black}{involutive} automorphism that is invariant on a complex submanifold of co-dimension $1$. Thus while the dominating domain is a classical Cartan domain, we lose no generality assuming that the deck automorphism is linear.

In fact the assumption that the non-trivial deck automorphism $F$ in the case of classical Cartan domains is a reflection, is not restrictive at all. This follows from the following observation.

\begin{remark}\label{remark:only-reflections} \it Note that in the situation of $2$-proper holomorphic mapping $\pi:D\to G$, where $D,G$ are domains in $\mathbb C^n$ and the set of fixed points of the non-trivial deck automorphism is non-empty and the analytic set $\mathcal J$ has dimension $n-1$.


Consequently, we lose no generality assuming only those $F$ which are linear \textcolor{black}{involutive} isomorphisms of classical Cartan domains having hyperplanes as their set of fixed points. Our next aim is to identify the domains that can be produced as images of corresponding $2$-proper holomorphic mappings from classical Cartan domains.
\end{remark}


\textcolor{black}{We recall following result from \cite[p. 18, Main Theorem]{Mes 1988} modified suitably for our purpose.
\begin{theorem}\cite[p. 18, Main Theorem]{Mes 1988} \label{Meschiari}
For $n>1,$ every proper holomorphic map
from an irreducible bounded symmetric domain of classical type into $\mathbb C^n$ is isomorphic to any basic polynomial map associated to some finite pseudoreflection group. 
\end{theorem}
In other words, the deck automorphism group of any proper holomorphic mapping from an irreducible classical Cartan domain in $\mb C^n,\,\,n>1,$ to some domain in $\mb C^n$ is a finite pseudoreflection group.} Our next proposition states that it is enough to identify pseudoreflections (or reflections) on irreducible classical Cartan domains up to a conjugation in order to determine such proper holomorphic mappings (or $2$-proper holomorphic mappings) up to an isomorphism.

\begin{proposition} \label{isomorphic}
Let $D\subset \mb C^n,\,\,n>1,$ be an irreducible classical Cartan domain. Suppose that $\phi_1: D\to D_1$ and $\phi_2:D \to D_2$ be two proper holomorphic mapping with the deck automorphism groups $\Gamma_1$ and $\Gamma_2,$ respectively. Then $\phi_1$ is isomorphic to $\phi_2$ if and only if $\Gamma_1$ is a conjugate to $\Gamma_2$ in $\Aut(D).$
\end{proposition}

\begin{proof}
If $\Gamma_1$ is a conjugate to $\Gamma_2$ then the result follows from Proposition \ref{proposition:conjugate}. 

Suppose that $\phi_1$ is isomorphic to $\phi_2,$ that is, there exist $h \in \Aut(D)$ and $\psi:D_2\to D_1$ biholomorphism such that $$\phi_1\circ h = \psi_\circ \phi_2.$$ Using \cite[p. 18, Main Theorem]{Mes 1988} (cf. Theorem \ref{Meschiari}), we get the following representation of $\phi_1$ : \begin{eqnarray}
\phi_1 = \psi_1 \circ \theta_1 \circ h_1
\end{eqnarray}
where $\theta_1$ denotes the basic polynomial mapping associated to $\Gamma_1,$ $\psi_1 : \theta_1(D) \to D_1$ is a biholomorphism and $h_1 \in \Aut(D).$ Then $\phi_2 = \psi^{-1}\circ  \phi_1\circ h = \psi^{-1}\circ  \psi_1 \circ \theta_1 \circ h_1\circ h.$ Clearly, $h_2 = h_1\circ h \in \Aut(D)$ and the deck automorphism group of $\phi_2$ is given by $h_2^{-1}\Gamma_1 h_2$ which is $\Gamma_2$ by assumption. This completes the proof.
\end{proof}

In particular, if $\phi_1$ and $\phi_2$ are $2$-proper holomorphic mappings then $\Gamma_i$'s are reflection groups (cf. Remark~\ref{remark:only-reflections}). Recall that in \cite[p. 703, Theorem 3]{Got 1969}, a classification of all reflections, unique up to a conjugation in $\Aut(D),$ is given for irreducible classical Cartan domains $D.$ Following that classification, we enlist below basic polynomial maps associated to the groups which are generated by those reflections in $D.$ Each proper holomorphic mapping of the list is non-isomorphic to each other. 

\textcolor{black}{\begin{proposition}\label{list}
    Let $D$ be an irreducible classical Cartan domain and $G$ be a domain in $\mb C^n.$ A  proper holomorphic mapping $\phi: D \to G$ of multiplicity $2$ is isomorphic (as in Definition \ref{iso}) to exactly one of the following:
    \begin{enumerate}
        \item[1.] $\phi_1 :\mathbb D\to\mathbb D$ such that $\phi_1(\lambda)=\lambda^2.$
        \item[2.] For $n>1,$ the $2$-proper holomorphic mapping $\phi_2 :\mathbb B_n \to \mathcal E(1/2,1,\ldots,1)$ which is defined by $$ \phi_2(z)= (z_1^2,z_2,\ldots,z_n),$$ where $\mathcal E(1/2,1,\ldots,1):=\{z\in\mathbb C^n:|z_1|+|z_2|^2+\ldots+|z_n|^2<1\}$ denotes the complex ellipsoid. 
        \item[3.] The $2$-proper holomorphic mapping $\phi_3:\mathcal R_{III}(2)\to\mathbb E$ which is defined by \begin{eqnarray}\label{rf}\phi_3(A)=(a_{11},a_{22}, a_{11}a_{22}-a^2)\,\,\, \text{ for } A = \begin{pmatrix}
          a_{11} & a\\
          a & a_{22}
        \end{pmatrix} \in \mathcal R_{III},\end{eqnarray} where  $\mb E:=\phi_3(\mathcal R_{III}(2))$ refers to the tetrablock.
        \item[4.]  The $2$-proper holomorphic mapping $\phi_4:\mathcal R_I(2\times 2)\to\phi_4(\mathcal R_I(2\times 2)):=\mathbb F$ which is defined by $$\phi_4(A) = (a_{11},a_{22}, a_{11}a_{22}-a_{12}a_{21},a_{12}+a_{21})\,\,\, \text{ for } A = \begin{pmatrix}
          a_{11} & a_{12}\\
          a_{21} & a_{22}
        \end{pmatrix} \in \mathcal R_I(2\times 2).$$
        \item[5.] For $n\geq 5,$ the $2$-proper holomorphic mapping $\Lambda_n:L_n \to \Lambda_n(L_n):=\mathbb L_n$ which is defined by \begin{eqnarray}\label{rfff} \Lambda_n(z)= (z_1^2,z_2,\ldots,z_n).\end{eqnarray}
    \end{enumerate}
\end{proposition}    }
    


\textcolor{black}{ \begin{proof}
    As a particular case of \cite[p. 18, Main Theorem]{Mes 1988} (cf. Theorem \ref{Meschiari}), we get that a proper holomorphic mapping $\phi$ of multiplicity $2$ from an irreducible classical Cartan domain is isomorphic to the basic polynomial of maps of some reflection group of order $2.$ So it is enough to find reflections in $\Aut(D),$ unique up to a conjugation in $\Aut(D)$ (cf. Proposition \ref{isomorphic}). Invoking Gottschling's result \cite[p. 703, Theorem 3]{Got 1969} for irreducible classical Cartan domains $D,$ a classification of all reflections in $\Aut(D)$ can be obtained. This provides the above list of irreducible classical Cartan domains $D$ such that there exist reflections in $\Aut(D)$.  The above list contains basic polynomial maps associated to the groups generated by those reflections.
    Combining it with Theorem \ref{Meschiari}, we conclude that the proper holomorphic map $\phi$ has to be isomorphic to at least one of them and since those are non-isomorphic to each other, the result follows. 
\end{proof}}
\textcolor{black}{\begin{remark}
\begin{enumerate}
    \item  Since there does not exist any reflection on the exceptional Cartan domains of dimensions $16$ and $27$ (cf. \cite{Mes 1972}), the above list is exhaustive for all irreducible Cartan domains.
    \item Results in similar direction as in Proposition \ref{list} can be found in \cite[p. 33, Corollary]{Mes 1988} and \cite{Mok-Ng-Tu 2010}. 
\end{enumerate}
\end{remark} }

We use Proposition \ref{list} and Gottschling's result to prove the following for reducible Cartan domains.
\begin{proposition}
   \textcolor{black}{ Let $D$ be a bounded symmetric domain such that $D= D_1^{k_1} \times \cdots \times D_r^{k_r}$ for non-equivalent irreducible Cartan domains $D_i : i=1,\ldots,r$ and $G$ be a domain in $\mb C^n.$ Suppose that there exists a proper holomorphic mapping $\phi: D \to G$ of multiplicity $2$ then $G$ is biholomorphic to exactly one of the following:
\begin{enumerate}
    \item[1.] $\text{ the unit disc } \mb D,$
    \item[2.] $\text{ the complex ellipsoid }\mathcal E(1/2,1,\ldots,1),$
    \item[3.] $\text{ the tetrablock } \mb E,$
    \item[4.] the domain $\mb F,$
    \item[5.] the domains $\mb L_n, n\geq 5,$
    \item[6.] the symmetrized bidisc $\mb G_2\text{ and,}$
    \item[7.] the Cartesian product of any one of the above domains and a bounded symmetric domain.
\end{enumerate} }
\end{proposition}

 \begin{proof}
 \textcolor{black}{Recall from Proposition \ref{proposition:idempotent-automorphism} and Corollary \ref{galois} that if there exists a $2$-proper holomorphic mapping $\phi: D \to G$ then there exists an involutive automorphism $g$ in $\Aut(D)$ which fixes a hyperplane. Moreover, ${\rm Deck}(\phi)$ is the group generated by $g$ and $G$ is biholomorphic to $D/{\rm Deck}(\phi)$ and thus $G$ is biholomorphic to the image of $D$ under the basic polynomial map of ${\rm Deck}(\phi)$ by analytic Chevalley-Shephard-Todd theorem. Now we mention the only possibilities of existing such an element $g$ in $\Aut(D)$ for the bounded symmetric domains $D$ of afore-mentioned form.}
 
     \textcolor{black}{Note that $\Aut(D)$ is given by the semi-direct product of $\mathfrak S(k_1) \times \cdots \times \mathfrak S(k_r)$ and  $\Aut(D_1)^{k_1} \times \cdots \times \Aut(D_r)^{k_r},$ where $\mathfrak S(n)$ denotes the permutation group on $n$ symbols \cite[p. 702, Theorem 1]{Got 1969}. From \cite[p. 702, Theorem 2; p. 703]{Got 1969}, we observe that such involutive automorphism $g$ (reflection) can occur in $\Aut(D)$ in the only two manners: $1)$ if none of $D_i$'s is biholomorphic to the unit disc $\mb D$ in the complex plane, the reflection can occur in $\Aut(D_1)^{k_1} \times \cdots \times \Aut(D_r)^{k_r}$ and such a reflection must be the identity on all but one of the factors of $D$. $2)$ Suppose that one of $D_i$ is biholomorphic to $\mb D,$ then if $k_i=1$ then a reflection can occur only in $\Aut(D_i)$ and if $k_i\geq 2$ then additionally a reflection can occur in $\Aut(D_i)^2$ as well. We enlist below basic polynomial maps associated to the groups generated by those reflections. \begin{enumerate}
    \item[1.] If exactly one (say $D_{1}$) is biholomorphic to $\mathbb D$ with $k_{1}=1$ or none of $D_i$'s is biholomorphic to the unit disc $\mathbb D$ but at least one $D_i$ is biholomorphic to any domain considered in 2-5 then any $2$-proper holomorphic mapping $\widetilde{\Phi} : D \to G$ is isomorphic to $$(w,z)\to (\Phi(w),z)\in \Phi(D_1)\times D_1^{k_1-1} \times \cdots \times D_r^{k_r} ,$$ where $D_1$ is any domain considered in 1-5 of Proposition \ref{list} and $\Phi$ is the corresponding $2$-proper holomorphic mapping.
     \item[2.] Recall the symmetrization map defined by \begin{eqnarray}\label{rff} \pi:\mathbb D^2\owns z\to (z_1+z_2,z_1z_2)\in\mathbb G_2.\end{eqnarray} If some $D_i$  is biholomorphic to $\mathbb D$ with $k_i\geq 2,$ then a $2$-proper holomorphic mapping $\widetilde{\Phi} : D \to G$ is isomorphic to either \begin{eqnarray*}
        (w,z)&\to& (\pi(w),z)\in \mb G_2\times D_1^{k_1-2}\times \cdots \times D_r^{k_r}\end{eqnarray*} where $D_1$ is biholomorphic to $\mb D$ or \begin{eqnarray*}
        (w,z)\to (\Phi(w),z)\in \Phi(D_1)\times D_1^{k_1-1} \times \cdots \times D_r^{k_r}
    \end{eqnarray*} where $D_1$ is any domain considered in 1-5 of Proposition \ref{list} and $\Phi$ is the corresponding $2$-proper holomorphic mapping.
\end{enumerate} Thus we get the above list of possible images of $2$-proper holomorphic mappings from bounded symmetric domains (up to biholomorphisms).}\end{proof}

\begin{remark}\label{clear} The unit disc $\mathbb D$ and the complex ellipsoid $\mathcal E(1/2,1,\ldots,1)$ are standard domains in literature. Recently, the geometry of $\mathbb G_2$ and $\mathbb E$ has been extensively studied and a lot is known. Motivated by those, a similar study is conducted to verify the same properties for the new domain $\mathbb F$ and the family of domains $\mathbb L_n,\,n\geq 5$. Nevertheless, before doing that, below we establish biholomorphisms between lower dimensional $\mb L_n$ for $n=2,3,4$ (defined appropriately) and $\mathbb G_2$, $\mathbb E,\mathbb F$ respectively. Thus our study can be reduced to the family of domains $\mathbb L_n$, $n\geq 2,$ since it includes all the domains which are not standard and arising in the above list. 
\end{remark}

\subsection{Lower dimensional domains}
It follows from Siegel's result that the lower dimensional Lie balls are biholomorphic to other classical domains. More precisely, $L_2$ is biholomorpic to $\mathbb D^2$, $L_3$ is biholomorphic to $\mathcal R_{III}(2)$ and $L_4$ is biholomorphic to $\mathcal R_I(2\times 2)$. This observation is crucial in the reduction of the discussion of the $2$-proper holomorphic images of classical Cartan domains to the case of the Lie balls. We therefore present the explicit transformations of the domains in the following lemma which is also a direct generalization of \cite[Lemma 3]{Coh-Col 1994}.

\begin{lemma}\label{bimap}
    Let $z\in\mathbb C^n$. Then the following hold :

    \begin{itemize} 
    \item   $z\in L_2$ if and only if $(z_1+iz_2,-z_1+iz_2)\in\mathbb D^2$,\\
    \item  $z\in L_3$ if and only if $\begin{bmatrix} z_1+iz_2 & z_3 \\ z_3 & -z_1+iz_2\end{bmatrix}\in\mathcal R_{III}(2)$,\\
    \item  $z\in L_4$ if and only if $\begin{bmatrix} z_1+iz_2 & z_3+iz_4 \\ z_3-iz_4 & -z_1+iz_2\end{bmatrix}\in\mathcal R_{I}(2\times 2)$.
    \end{itemize} 
\end{lemma}
\begin{proof}
The case $n=2$ is done in \cite[Lemma 3(b)]{Coh-Col 1994}. For $n=3$ and $n=4,$ we consider a $2\times 2$ complex matrix $A$ and remark that $A$ is such that $\mathbb I_2-AA^*>0$ if and only if
\begin{align}\label{positive-definite-1}
2>&|a_{11}|^2+|a_{12}|^2+|a_{21}|^2+|a_{22}|^2,\\
    1+|a_{11}a_{22}-a_{12}a_{21}|^2>&|a_{11}|^2+|a_{22}|^2+|a_{12}|^2+|a_{21}|^2.\label{positive-definite-2}
\end{align}
Now we make the substitution $A=\begin{bmatrix} z_1+iz_2 & z_3+iz_4 \\ z_3-iz_4 & -z_1+iz_2\end{bmatrix}$ to get that the formulae in Equation \eqref{positive-definite-1} and Equation \eqref{positive-definite-2} which read as
\begin{align*}
 1>&||z||^2,\\
    1+|z\bullet z|^2&> 2||z||^2,   
\end{align*}
    which proves the case $n=4$. Substituting $z_4=0,$ we get the result for $n=3.$
\end{proof}
A direct consequence of Lemma \ref{bimap}  provides formulae for explicit biholomorphisms (permutation of coordinates is required in some cases) of $\mathbb G_2$ to $\mathbb L_2$, $\mathbb E$ to $\mathbb L_3$ and $\mathbb F$ to $\mathbb L_4$. 
\begin{corollary}\label{bimap2} The following formulae provide biholomorphisms between the corresponding domains:
\begin{itemize}
    \item $\mathbb L_2\owns w\to (2iw_2,-w_1-w_2^2)\in \mathbb G_2$,
    \item $\mathbb L_3\owns w\to (w_2+iw_3,-w_2+iw_3,-w_3^2-w_2^2-w_1)\in\mathbb E$,
    \item $\mathbb L_4\owns w\to (w_3+iw_4,-w_3+iw_4,-w_2^2-w_3^2-w_4^2-w_1,2w_2)\in\mathbb F$.
    \end{itemize}
    \end{corollary}
The above formula gives another way of representing the symmetrized bidisc, see \cite{Agl-You 2001} or a recent paper \cite{Bha-Bis-Mai 2022} where an unexpected link between the symmetrized bidisc and the domain from Isaev's list of Kobayashi hyperbolic domains with big group of automorphisms is established. In other words, to study these newly defined domains along with the symmetrized bidisc and the tetrablock, it is enough to consider $\mathbb L_n$, $n\geq 2$. Though it is not required to introduce the domain $\mathbb F$ separately, it seems that this representation of $\mathbb L_4$ is easier to construct many objects and properties, for example, Carath\'eodory extremal functions or the proof of $\mathbb C$-convexity of $\mathbb F$.

\section{Geometry of $\mathbb L_n$}
This section is dedicated on the study of the complex geometry of $\mathbb L_n$. Evidently, $w\in\mathbb L_n$ if and only if $w\in\mathcal E(1/2,1\ldots,1)$ and 
\begin{equation*}
\sqrt{\left(|w_1|+\sum_{j=2}^n|w_j|^2\right)^2-\left|w_1+\sum_{j=2}^nw_j^2\right|^2}<1-\left(|w_1|+\sum_{j=2}^n|w_j|^2\right).
\end{equation*}
Let $k_j,\,j=1,\ldots,n$ be positive integers. A domain $D\subset \mathbb C^n$ is said to be a {\it $(k_1,\ldots,k_n)$-balanced} domain if for any $z\in D$ and $|\lambda|\leq 1,$ $(\lambda^{k_1}z_1,\ldots,\lambda^{k_n}z_n)\in D$ (cf. \cite{Nik 2006}). The domain $\mathbb L_n$ is a bounded hyperconvex $(2,1,\ldots,1)$-balanced domain.  
\subsection{Shilov boundary}
  Let $\partial_s D$ denote the Shilov boundary of the domain $D.$ 

\begin{proposition} The Shilov boundary of $\mathbb L_n$ is given by
\begin{equation*}
\partial_s \mathbb L_n=\Lambda_n(\partial_s L_n),
\end{equation*}
where $\Lambda_n$ is as defined in Equation \eqref{rfff}.
\end{proposition}
\begin{proof}
\textcolor{black}{Inferring Equation \eqref{rfff}, we obtain that $\Lambda_n : \mb C^n \to \mb C^n$ is a proper holomorphic mapping. Then if $L$ is a domain relatively compact in $\mb C^n$ and $K = \Lambda_n^{-1}(L),$ from \cite[Theorem 3]{Kos 2011} we get $\partial_s L = \Lambda_n (\partial_s K).$ We consider $K=\overline{L_n}$ and $ L=\Lambda_n(K)$ then the above description of $\partial_s \mathbb L_n$ is a direct consequence of \cite[Theorem 3]{Kos 2011}.}
\end{proof}
Since the Shilov boundary of any classical Cartan domain is well-known  \cite{Hua 1963}, the above description is complete.
\subsection{(Non-)convexity} We prove below that $\mathbb L_n$ cannot be exhausted by domains biholomorphic to convex ones. This extends the results of \cite{Cos 2004, Edi 2004, Edi-Kos-Zwo 2013} on the symmetrized bidisc and the tetrablock.

\textcolor{black}{From the definition of $\mathbb L_n$, we get $\mathbb L_2\times\{0\}^{n-2}=\{z\in\mathbb L_n:z_j=0,\;j=3,\ldots,n\}.$ In \cite[Lemma 3]{Coh-Col 1994}, it is shown that if $(z_1,\ldots,z_k) \in L_k,$ then for $m<k,$ $(z_1,\ldots,z_m) \in L_m.$ Thus the projection of $\mathbb L_n$ on $\mathbb C^2\times\{0\}^{n-2}$ is $\mathbb L_2\times\{0\}^{n-2}.$ It is crucial in the proof of the next result.}

\textcolor{black}{For a domain $D,$ we denote by $S(D)$ the set of all holomorphic mappings $F:D\times D \to D$ such that $F(z,z)=z$ and $F(z,w) = F(w,z)$ for $z,w \in D.$ Moreover, the property $S(D) = \varnothing$ is biholomorphically invariant. In \cite[Corollary 3]{Edi 2013}, Edigarian proved that for the symmetrized bidisc $\mb G_2,$ $S(\mb G_2) = \varnothing.$  Since $\mathbb L_2$ is biholomorphic to $\mathbb G_2,$ $S(\mathbb L_2\times\{0\}^{n-2}) = \varnothing.$ Now we directly apply  \cite[Proposition 2.1]{Zap 2015} to prove the following.}

\begin{theorem}\label{convex} The domain $\mathbb L_n$ cannot be exhausted by domains biholomorphic to convex ones, $n\geq 2$. 
\end{theorem}
\begin{proof}  \textcolor{black}{From the above discussion, we get that $\{(z_1,z_2): (z_1,\ldots,z_n)\in \mathbb L_n\}=\mathbb L_2$ and $\mathbb L_2\times\{0\}^{n-2}$ is a holomorphic retract of $\mathbb L_n.$ Also, it is taut and bounded with $S(\mathbb L_2\times\{0\}^{n-2})= \varnothing.$ Hence, from \cite[Proposition 2.1]{Zap 2015} we conclude that $\mb L_n$ cannot be exhausted by domains biholomorphic to convex ones.}
\end{proof} 

 In our proof, we rely on \cite[Proposition 2.1]{Zap 2015} which is a continuation of the results (and methods) introduced in \cite{Cos 2004} (which are later developed in \cite{Edi 2004, Edi 2013}).

\begin{remark} Recall that  both $\mathbb G_2$ and $\mathbb E$ are $\mathbb C$-convex domains \cite{Nik-Pfl-Zwo 2008, Zwo 2013}. Since these domains are biholomorphic to $\mathbb L_2$ and $\mathbb L_3$ it seems natural to pose a question whether all the domains $\mathbb L_n$ are biholomorphic to $\mathbb C$-convex domains. Perhaps the first step in answering the problem would be to show $\mathbb C$-convexity of $\mathbb F$ which is biholomorphic to $\mathbb L_4.$
\end{remark}

\subsection{Zeroes of the Bergman kernel and inhomogeneity} Another point of interest in complex analysis is to conclude whether the Bergman kernel of a domain has any zero. A closed form of the Bergman kernel is specifically useful for it. We use the transformation formula in \cite{Bel 1982} or a direct version in \cite{Gho 2021} for the Bergman kernels under a proper holomorphic mapping to find the Bergman kernel for $\mathbb L_n.$ Recall that the Bergman kernel of $L_n$ is given by \cite{Ara 1995}\begin{eqnarray*}
    K_{L_n}(z,w) = (1+\sum_{j=1}^n z_j^2\sum_{j=1}^n \overbar{w_j}^2-2\sum_{j=1}^n z_j\overbar{w_j} )^{-n}
\end{eqnarray*} and the Bergman kernel of $\mathbb L_n$ is \begin{eqnarray}\label{formla}
    \nonumber&&K_{\mathbb L_n}(\Lambda_n(z),\Lambda_n(w))\\ \label{ker} &=& \frac{1}{4z_1\overbar{w_1}}(K_{L_n}(z,w) - K_{L_n}(\sigma \cdot z,w))  \\ &=&\nonumber \frac{1}{4z_1\overbar{w_1}}((1+\sum_{j=1}^n z_j^2\sum_{j=1}^n \overbar{w_j}^2-2\sum_{j=1}^n z_j\overbar{w_j} )^{-n} - (1+\sum_{j=1}^n z_j^2\sum_{j=1}^n \overbar{w_j}^2-2\sum_{j=2}^n z_j\overbar{w_j} + 2z_1\overbar{w_1} )^{-n}),\end{eqnarray} 
 where the action of the reflection $\sigma$ is defined by $\sigma \cdot (z_1,z_2,\ldots,z_n) = (-z_1,z_2,\ldots,z_n).$   
 Using binomial expansion, we get the following:    
\begin{eqnarray}\label{formula}
    K_{\mathbb L_n}(\Lambda_n(z),\Lambda_n(w)) = \begin{cases} & \frac{\binom{n}{1}X_n^{n-1} +\binom{n}{3}X_n^{n-3} A^2 + \ldots + \binom{n}{1}X_n A^{n-2}}{(X_n^2 - A^2)^n}, \,\, n \text{ even;} \\ & \frac{\binom{n}{1}X_n^{n-1} +\binom{n}{3}X_n^{n-3} A^2 + \ldots + A^{n-1}}{(X_n^2 - A^2)^n}, \,\, n \text{ odd}; \end{cases}
\end{eqnarray} where $X_n = 1+\sum_{j=1}^n z_j^2\sum_{j=1}^n \overbar{w_j}^2-2\sum_{j=2}^n z_j\overbar{w_j} $  and $A = 2z_1\overbar{w_1}.$ This formula gives a closed form for the Bergman kernel $K_{\mathbb L_n}$.


    If the Bergman kernel of a domain $D$ has no zero in $D\times D,$ we say $D$ is {\it Lu Qi-Keng domain}. It is known that $\mathbb L_2$ is a Lu Qi-Keng domain \cite{Edi-Zwo 2005}. However, the result is otherwise for $n\geq 3.$ 
\begin{proposition}\label{Lu Qi-Keng}
   The domain $\mathbb L_n$ is not a Lu Qi-Keng domain for $n\geq 3.$
\end{proposition}

\begin{proof}
Let $z_0 = \frac{\omega_{2n} - 1}{\omega_{2n} + 1}$ for $\omega_{2n}=cos(\pi/n)+isin(\pi/n).$ 
For $n\geq 3,$ $0<|z_0| <1.$ Choose $r \in (0,1)$ such that $z_0/r \in \mathbb D.$ Note that for $\widehat{\bl z}=(z_0/r,0,\ldots,0)$ and $\widehat{\bl w}=(r,0,\ldots,0),$ we get $X_n = 1+z_0^2 $ and $A = 2z_0$ and thus \begin{eqnarray*}
  K_{L_n}(\widehat{\bl z},\widehat{\bl w})&=&  \frac{1}{(1-z_0)^{2n}}=\frac{(1+\omega_{2n})^{2n}}{2^{2n}} ,\\
K_{L_n}(\sigma \cdot \widehat{\bl z},\widehat{\bl w})&=& \frac{1}{(1+z_0)^{2n}}=\frac{(1+\omega_{2n})^{2n}}{(2\omega_{2n})^{2n}}.  
\end{eqnarray*}

Since $\omega_{2n}^{2n}=1,$ we get from Equation \eqref{formla} \begin{eqnarray*}
        K_{\mathbb L_n}(\Lambda_n(\widehat{\bl z}),\Lambda_n(\widehat{\bl w}))=0,
    \end{eqnarray*}
 which proves the result.
\end{proof}

\begin{remark} Trybula showed $\mathbb L_3$ is not a Lu Qi-Keng domain \cite{Try 2013}. The above observation is a generalization of it.
\end{remark}

Note that Equation \eqref{formula} shows that
\begin{equation}
    K_{\mathbb L_n}(0,z)=n
\end{equation}
for every $z \in \mathbb L_n.$ Therefore, an immediate consequence of Proposition \ref{Lu Qi-Keng} is the following result.
\begin{corollary}
    \label{homogennot}
    The domain $\mathbb L_n$ is inhomogeneous for $n \geq 3.$
\end{corollary}

\subsubsection{Holomorphic functions on $(k_1,\ldots,k_n)$-quasi balanced domains (An alternative attitude to Lu Qi-Keng problem and inhomogeneity of $\mathbb L_n$)} 
\textcolor{black}{Consider non-negative integers $k_1,\ldots,k_n$ and $N$. A complex polynomial $P$ in $n$ variables is called {\it $(k_1,\ldots,k_n)$-homogeneous} of degree $N$ if $P(\lambda^{k_1} z_1,\lambda^{k_2} z_2,\ldots,\lambda^{k_n} z_n)=\lambda^N P(z_1,z_2,\ldots,z_n),$ $(z_1,\ldots,z_n)\in\mathbb C^n$, $\lambda\in\mathbb C$.} Any $(k_1,\ldots,k_n)$-homogeneous polynomial of degree $N$ can be (uniquely) presented as the sum of monomials
\begin{equation*}
    P(z)=\sum_{k_1\alpha_1+\ldots k_n\alpha_n=N}c_{\alpha}z^{\alpha},
\end{equation*}
where the sum is taken over all $\alpha\in\mathbb N^n$ such that $k_1\alpha_1+\ldots+k_n\alpha_n=N$, $c_{\alpha}\in\mathbb C$ and $z^{\alpha}:=z_1^{\alpha_1}\cdot\ldots\cdot z_n^{\alpha_n}$.

Following the standard reasoning as in balanced domains, we obtain the expansion of holomorphic functions in $(k_1,\ldots,k_n)$-quasi balanced domains (compare e. g. \cite{Jak-Jar 2001}). 
\begin{theorem}
    Let $F$ be a holomorphic function defined on the $(k_1,\ldots,k_n)$-balanced pseudoconvex domain. Then 
    \begin{equation*}
    F(z)=\sum_{\nu=0}^{\infty}Q_{\nu}(z),
    \end{equation*}
    where $Q_{\nu}$ is a $(k_1,\ldots,k_n)$-homogeneous polynomial of degree $\nu$ and the convergence is locally uniform in $D$. 
    
    Additionally, if $F$ is $L^2$-integrable then $Q_{\nu}$ is from $L^2(D)$ the convergence is in $L^2$, polynomials $Q_{\nu}$ are orthogonal. 
    
    Moreover, there exists a complete orthonormal system of $(k_1,\ldots,k_n)$-homogeneous polynomials $\{P_j\}_j$ of the Bergman space  $L_h^2(D)$. Consequently, the Bergman kernel $K_D$ is given by the formula
    \begin{equation*}
        K_D(0,z)=\frac{1}{Vol(D)},\; z\in D.
    \end{equation*}
\end{theorem}
The above theorem shows that if $D$ is a bounded $(k_1,\ldots,k_n)$-quasi balanced transitive domain then it is Lu Qi-Keng domain, which gives another proof of Corollary~\ref{homogennot}.

\section{Proper holomorphic self-maps of $\mathbb L_n$ and the group $\Aut(\mathbb L_n)$} Two basic differences between $L_2$ and $L_n$, $n\geq 3$ are that the first domain is reducible and the existence of proper holomorphic self-maps which are not automorphisms. It turns out that the (non)-rigidity property of proper holomorphic self-maps is passed on to the domains $\mathbb L_n$ accordingly. Recall that the structure of proper holomorphic self-mappings is known for $\mathbb L_2$ and those are determined by the proper holomorphic self-mappings on the locus set (that is the unit disc) \cite{Edi-Zwo 2005, Edi 2004b}. For the tetrablock or $\mathbb L_3,$ all the proper holomorphic self-mappings are automorphisms \cite{Kos 2011} and the group of automorphisms of the tetrablock is determined in \cite{You 2008}. Moreover, the automorphisms are determined by automorphisms on the locus set (biholomorphic to $\mathbb D^2$ or $L_2$). We observe that the same phenomenon holds for $\mathbb L_n$ for all $n\geq 3$.

We note that the rigidity property of $\mathbb L_n$, $n\geq 3$ in the next proposition which is a direct consequence of \cite{Mok-Ng-Tu 2010}.

\begin{proposition}\label{alexander}
    Suppose that $n\geq 3.$ Every proper holomorphic map $\phi : \mathbb L_n \to \mathbb L_n$ is an automorphism. 
\end{proposition}
\begin{proof}
First we observe that $\phi \circ \Lambda_n : L_n \to \mathbb L_n$ is a proper holomorphic mapping. By \cite[p. 815, Theorem 1]{Mok-Ng-Tu 2010}, the multiplicity of $\phi \circ \Lambda_n$ is $2.$ Since the multiplicity of $\Lambda_n$ is also $2,$ it is evident that $\phi$ is an automorphism.
\end{proof}

In subsequent subsections, we determine the group of automorphisms of $\mathbb L_n$. Initially, we construct automorphisms of $\mathbb L_n$ extending from the locus set $\{0\}\times L_{n-1}$ and subsequently, prove that there are no other automorphisms. The latter follows from the inhomogeneity of $\mathbb L_n$. 

\subsection{Extension of automorphisms from the locus set to $\mathbb L_n$}
 We start with observing the construction of automorphisms of $\mathbb G_2$ and $\mathbb E$ more intrinsically. (All) the automorphisms of these two domains are induced by automorphisms of some lower dimensional sets. 
 Those lower dimensional sets are $\{(2\lambda,\lambda^2):\lambda\in\mathbb D\}$ (biholomorphic to $\mathbb D$) for $\mathbb G_2$ and the set $\{(a,b,ab):a,b\in\mathbb D^2\}$ (biholomorphic to $\mathbb D^2$) for $\mathbb E$ which are critical values of the proper holomorphic maps $\pi$ defined in Equation \eqref{rff} and $\phi_3$ defined in Equation \eqref{rf}, respectively. For $\mathbb L_n,$ we focus on the set of critical values of $\Lambda_n,$ that is, $\mathcal J_{\Lambda_n}:=\{0\}\times L_{n-1}$. We use it to produce automorphisms of $\mathbb L_n$ extending the automorphisms of $L_{n-1}$. Before doing that, we make one general observation. 

\begin{remark}\label{remark:extension} We consider a $2$-proper holomorphic mapping $\pi:D\to G$ between domains $D,G\subset\mathbb C^n$ with $\mathcal J_\pi:=\{z\in D:\operatorname{det}\pi^{\prime}(z)=0\}$. Assume that 
\begin{enumerate}
    \item[1.] $a$ is an automorphism of $D$ such that $a|_{\mathcal J_\pi}$ is an automorphism of $\mathcal J_\pi$ and
    \item[2.] $a$ preserves the fibers $\pi^{-1}(w)$, $w\in G,$
\end{enumerate}  then $a$ induces an automorphism on $G$ extending from $\pi(\mathcal J_\pi).$ We apply this observation to prove the following result. 
\end{remark} 

\begin{theorem}\label{theorem:extension} Any automorphism of the locus set $\mathcal J_{\Lambda_{n+1}}=\{0\}\times L_{n}$ extends to an automorphism of $\mathbb L_{n+1}$, $n\geq 1$.
 \end{theorem} \begin{proof} Recall that the automorphisms of all classical Cartan domains are generated by the linear isomorphisms of the domain and the \textcolor{black}{involutions} which intertwine some point of the domain and the origin. It is required to verify the conditions described in Remark~\ref{remark:extension} for these two classes of automorphisms of $\{0\}\times L_{n}\subset L_{n+1}$.

The linear automorphisms of $L_{n}$ are of the form 
\begin{equation*}
L:\mathbb C^n\owns z\to\omega ~U(z)\in\mathbb C^n,
\end{equation*}
where $|\omega|=1$ and $U$ is a special orthogonal matrix. It induces an isomorphism of $L_{n+1}$ by the formula
\begin{equation*}
\mathbb C^{n+1}\owns (z_1,z)\to \omega \big(z_1,U(z)\big)\in\mathbb C^{n+1},
\end{equation*}
which satisfies the properties in Remark~\ref{remark:extension}.

Now we investigate the same for an \textcolor{black}{involutive} automorphism of $L_n$ which maps some point of $L_n$ to the origin. 
Let $g = \begin{pmatrix}
    A&B\\C&D
\end{pmatrix},$ where $A,B,C$ and $D$ are associated matrices in the definition of the automorphism $\Psi_g$ of $L_n$ in \cite[p. 162]{deF 1991}, see also \cite[p. 86]{Hua 1963}. We define the corresponding matrices $\tilde A,\tilde B,\tilde C,\tilde D$ defining the automorphism $\Psi_{\tilde g}$ of $L_{n+1}$ with $\tilde z:=(z_1,z^{\prime})=(z_1,z_2,\ldots,z_{n+1})$ as following: 
\begin{equation*}
\tilde A:=
\begin{bmatrix} 
1 & 0_n \\
0^n & A
\end{bmatrix},
\,\tilde B:=
\begin{bmatrix}
0_2 \\
B
\end{bmatrix},\,
\tilde C:=
\begin{bmatrix}
0^2 & C
\end{bmatrix},\,
\tilde D:=D,
\end{equation*}
where $0_n$ denotes the $1\times n$ row vector with all $0$ entries and $0^n$ denotes the $n\times 1$ column vector with all $0$ entries.
The matrices satisfy the conditions in \cite[p. 162]{deF 1991} and the corresponding automorphism $\Psi_{\tilde g}$ of $L_{n+1}$ satisfies 
\begin{equation*}
\Psi_{\tilde g}(-z_1,z^{\prime})=(-(\Psi_{\tilde g}(z_1,z^{\prime}))_1,(\Psi_{\tilde g}(z_1,z^{\prime}))_{(2,\ldots,n+1)}).
\end{equation*}
Additionally, it extends the automorphism $(0,\Psi_g)$ of \textcolor{black}{$\Lambda_{n+1}(\{0\}\times L_n)=\{0\}\times L_n$}. Consequently, Remark~\ref{remark:extension} lets us to define an automorphism of $\mathbb L_{n+1}$.
\end{proof}

\begin{remark}\label{remark:cartan}  In order to show that each automorphism of $\mathbb L_n$ leaving the set \textcolor{black}{$\Lambda_{n}(\{0\}\times L_{n-1})=\{0\}\times L_{n-1}$} invariant is uniquely determined by the automorphism of $L_{n-1},$ it is sufficient to prove the next lemma. That lemma is a kind of Cartan theorem. The main ingredients of the proof are the above construction of the automorphisms of $\mathbb L_n$ from a lower dimensional submanifold and the transitivity of the automorphisms of Lie balls. We follow the ideas from the proof of \cite[Lemma 4.2]{You 2008} where the method is adapted from \cite{Jar-Pfl 2004}.
\end{remark}

\begin{lemma}\label{lemma:uniqueness}  For $n\geq 2,$ if $F\in\Aut(\mathbb L_n)$ and $F(0,z^{\prime})=(0,z^{\prime})$ for $z^{\prime}\in L_{n-1}$,  then $F$ is the identity.
    \end{lemma}
\begin{proof} For $|\omega|=1,$ we define $\rho_{\omega}:\mathbb L_n\owns z\to(\omega^2z_1,\omega z_2,\ldots,\omega z_n)\in\mathbb L_n$ and the function
    \begin{equation*}
        H_{\omega}(z):=F^{-1}(\rho_{\overbar{\omega}}(F(\rho_{\omega}(z))), \; z\in\mathbb L_n.
        \end{equation*}
        By assumption, we get that
    \begin{equation*}
    F^{\prime}(0)=
    \begin{bmatrix}
        a_1 & a_2 & \cdots & a_n\\
        0^{n-1} & &  \mathbb I_{n-1} &
    \end{bmatrix}.
    \end{equation*}
    Note that $a_1\neq 0$. Direct calculations lead to the following equality
    \begin{equation*}
        H_{\omega}^{\prime}(0)=
        \begin{bmatrix}
            1 & \frac{a_2(\overbar{\omega}-1)}{a_1} & \cdots & \frac{a_n(\overbar{\omega}-1)}{a_1}\\
            0^{n-1} & & \mathbb I_{n-1} & 
        \end{bmatrix}.
    \end{equation*}
    Consequently, for any $N\geq 1$ we have
    \begin{equation*}
    (H_{\omega}^{\prime}(0))^N=
      \begin{bmatrix}
            1 & \frac{Na_2(\overbar{\omega}-1)}{a_1} & \cdots & \frac{Na_n(\overbar{\omega}-1)}{a_1}\\
            0^{n-1} & & \mathbb I_{n-1} & 
        \end{bmatrix}.  
    \end{equation*}
    But the domain $\mathbb L_n$ is bounded so the matrices $(H_{\omega}^{\prime}(0))^N$ are uniformly bounded for $N\geq 1$. Therefore, we get that $\frac{\partial F_1}{\partial z_j}(0)=a_j=0$, $j=2,\ldots,n$.
   We write
   \begin{equation*}
       F^{\prime}(0)=\begin{bmatrix}
           a_1 & 0 & \cdots & 0\\
           0^{n-1} & & \mathbb I_{n-1} & 
       \end{bmatrix},
   \end{equation*}
   where $0<|a_1|\leq 1$. 

We prove that $a_1=1$. We denote the generalized Minkowski functional of $\mathbb L_n$ by $M$ (cf. \cite{Nik 2006}). The domain $\mathbb L_n$ is a $(2,1,\ldots,1)$-balanced domain and $M$ is logarithmically plurisubharmonic. For a fixed $z\in \mathbb L_n,$  consider the following function
   \begin{equation*}
       f:\lambda\to \left( \frac{F_1(\lambda^2 z_1,\lambda z_2,\ldots, \lambda z_n)}{\lambda^2}, \frac{F_2(\lambda^2 z_1,\lambda z_2,\ldots, \lambda z_n)}{\lambda}, \ldots, \frac{F_n(\lambda^2 z_1,\lambda z_2,\ldots, \lambda z_n)}{\lambda}\right).
   \end{equation*}
   Note that as $|\lambda|\to 1^-,$ we get that $\limsup_{|\lambda|\to 1} M(f(\lambda))\leq 1$. Since $\frac{\partial F_1}{\partial z_j}(0)=0$, $j=2,\ldots,n$, we get that $f$ at $0$ has the limit equal to
   \begin{equation*}
       (a_1 z_1+P(z^{\prime}),z_2,\ldots,z_n),
   \end{equation*}
   where $P$ is a homogeneous polynomial of degree two (and independent of $z_1$).
   The maximum principle for subharmonic functions gives us that the above expression belongs to $\overline{\mathbb L_n}$. Now consider the function
   \begin{equation*}
   \mathbb L_n\owns z \to (a_1z_1+P(z^{\prime}),z^{\prime})\in\overline{\mathbb L_n},
   \end{equation*}
and since the value of this function at $0$ is $0\in\mathbb L_n,$  the values of this function lie actually in $\mathbb L_n$. We take $N$-iteration of the above mapping which equals to $(a_1^Nz_1+NP(z^{\prime}),z^{\prime})\in\mathbb L_n$. It implies that $P(z^{\prime})=0$, $z^{\prime}\in L_{n-1}$, which shows that $(a_1z_1,z^{\prime})\in \mathbb L_n$ for any $(z_1,z^{\prime})\in \mathbb L_n.$ Hence, $a_1=1$. 

Since $F^{\prime}(0)=\mathbb I_n,$ Cartan theorem implies that $F$ is the identity.
\end{proof}

\begin{theorem}\label{theorem:automorphisms}
        $\Aut(L_{n-1}) \cong \Aut(\mathbb L_n)$ for $n\geq3.$
\end{theorem}

\begin{proof}
\textcolor{black}{We showed in Theorem~\ref{theorem:extension} that every automorphism $h$ of $\{0\} \times L_{n-1}$ extends to an automorphism $F_h$ of $\mathbb L_n.$  Let $A_n =\{F_h \in \Aut(\mathbb L_n): h \in \Aut(L_{n-1}) \}.$} 

    Now we prove that each automorphism of $\mathbb L_n$ gives rise to an automorphism of $L_{n-1}.$ 
    Let us denote $\bold 0 := (0,\ldots,0) \in L_n$ and $V:=\{F(\bold 0) : F \in \Aut(\mathbb L_n)\}.$ \textcolor{black}{We note that $\mathcal J_{\Lambda_n} =  \{0\}\times L_{n-1}$ and $\Lambda_n(\mathcal J_{\Lambda_n}) = \mathcal J_{\Lambda_n}$ which is a subset of $\mathbb L_n.$  $A_n$ acts transitively on $\Lambda_n(\mathcal J_{\Lambda_n}),$ thus for every $\bl z \in \Lambda_n(\mathcal J_{\Lambda_n})$ there exists an automorphism $h_{\bl z}$ in $A_n \subseteq \Aut(\mathbb L_n)$ such that $h_{\bl z}(\bl 0) = \bl z.$ Therefore, $\{0\}\times L_{n-1}=\Lambda_n(\mathcal J_{\Lambda_n}) \subseteq V.$ Also, $\Aut(\mathbb L_n)$ acts transitively on $V.$  Since $\Aut(\mathbb L_n)$ does not act transitively on $\mathbb L_n$ (cf. Corollary~\ref{homogennot}), so $V \subsetneq \mathbb L_n.$ The set $V$ is a closed connected complex submanifold of $\mathbb L_n$ \cite[Proposition 1]{Kaup 1970} and $\{0\}\times L_{n-1} = \Lambda_n(\mathcal J_{\Lambda_n}) \subseteq V\subsetneq \mathbb L_n.$ } Hence, $V=\{0\}\times L_{n-1}.$ 
    
   \textcolor{black}{ We take $(0,z_2,\ldots,z_n) \in \{0\}\times L_{n-1}=\Lambda_n(\mathcal J_{\Lambda_n})$ and $(0,z_2,\ldots,z_n) = h(\bold 0)$ for some $h \in A_n.$ Now for all $F \in \Aut(\mathbb L_n),$ $F(0,z_2,\ldots,z_n)= F\circ h(\bold 0) \in V = \{0\}\times L_{n-1}.$ Thus we conclude $F(\Lambda_n(\mathcal J_{\Lambda_n})) = \Lambda_n(\mathcal J_{\Lambda_n})$ for every $F \in \Aut(\mathbb L_n).$} Then for every $F \in \Aut(\mathbb L_n),$ the restriction $F|_{\Lambda_n(\mathcal J_{\Lambda_n})}$ is an automorphism of $\Lambda_n(\mathcal J_{\Lambda_n}).$ Then the result follows from Lemma~\ref{lemma:uniqueness} and Remark~\ref{remark:cartan}.
\end{proof}

\section{Concluding remarks and open questions}\label{conclu} Let us summarize the problems which are potent to be the subject of future research. Some of those
 have already been mentioned in the paper.
\begin{remark} 
\begin{enumerate}
\item[1.]
It is tempting that the Lempert theorem could be proven (at least in the case when one of the points is $0$) for $\mathbb L_n$ (compare \cite{Abo-You-Whi 2007}). Probably the general case of the Lempert theorem would be much more difficult (compare the proof of the case $n=3$ in \cite{Edi-Kos-Zwo 2013}). One of the possible methods would be a construction of candidates for left inverses (some rational functions?).

\item[2.] Another problem would be to find a $\mu$-synthesis approach that would produce as a result the domains $\mathbb L_n$ as it is the case of the symmetrized bidisc and the tetrablock.

\item[3.] Solve the $3$-point Nevanlinna-Pick problem for classical Cartan domains or at least for the domains $L_n$ (compare \cite{Kos 2015} and \cite{Kos-Zwo 2018}). 

\item[4. ] Can the domains $\mathbb L_n$ (or their biholomorphic images) play an essential role in the study of (complete) spectral sets or in other areas of the operator theory as the symmetrized bidisc and the tetrablock do?  

\item[5. ] Another problem is to provide a geometric characterization of the domain $\mathbb L_n$ by the structure of its automorphism group (which is equal to $\operatorname{Aut}(L_{n-1})$) and some other geometric assumptions. Recall that a similar characterization can be found in \cite{Agl-Lyk-You 2019} for the symmetrized bidisc (that is $\mathbb L_2$).  
\end{enumerate}
\end{remark}

\subsection*{Acknowledgements} The authors are indebted to the anonymous referees for a careful reading of the manuscript and for suggesting a detailed list of valuable changes which helped in improving the presentation of the article.

\end{document}